\documentclass[a4paper]{article}

\usepackage{
amsmath,
amsthm,
amscd,
amssymb,
}
\usepackage[
            left=3cm,right=3cm,top=3cm,bottom=3cm,%
            ]{geometry}
\usepackage{wasysym}
\usepackage{comment}
\usepackage{tikz}
\usetikzlibrary{positioning,arrows,calc}
\usepackage{xspace}

\usepackage{graphicx}

\usepackage{thm-restate}

\usepackage{url}

\theoremstyle{plain}

\newtheorem{definition}{Definition}

\newtheorem{theorem}{Theorem}

\newtheorem{question}{Question}

\newtheoremstyle{derp}
{3pt}
{3pt}
{}
{}
{\upshape}
{:}
{.5em}
{}
\theoremstyle{derp}

\newcommand{\R}{\mathbb{R}}

\newcommand{\Z}{\mathbb{Z}}

\newcommand{\N}{\mathbb{N}}
\newcommand{\M}{\mathbb{M}}

\newcommand{\ID}{\mathrm{id}}

\newcommand{\MCG}{\mathrm{MCG}}

\title{Veelike actions and the MCG of a mixing SFT}

\author{
Ville Salo \\
vosalo@utu.fi
}

\begin{document}
\maketitle

\begin{abstract}
We embed Thompson's group $V$ in the mapping class group of a mixing subshift of finite type. Question~6.3 in [Boyle-Chuysurichay, 17] asks whether these mapping class groups are sofic. Our result suggests that this question is difficult to solve at present (at least for some mixing SFTs), since a resolution of it in either direction would solve an open problem in geometric group theory. More generally, we define the notion of a ``veelike action'', and prove that whenever a group acts veelike on a language, it embeds in the mapping class group of some subshift. We show that Thompson's $V$ acts veelike on a locally testable language, and thus it embeds in the mapping class group of a mixing SFT. A two-sided variant of the argument works for the Brin-Thompson group $2V$.
\end{abstract}

\section{Introduction}

The mapping class group of a subshift was introduced in \cite{Bo02a}, and studied extensively in \cite{BoCh17} (also see \cite{ScYa21}). By embedding suitable groups as subgroups, it was in particular shown that this group is non-amenable and not residually finite. It was asked in \cite[Question~6.3]{BoCh17} whether this group is sofic. We show that solving this question would necessarily solve an open problem: If the group is shown sofic, then so is Thompson's group $V$ (and even the Brin-Thompson group $2V$), which is open. If the group is shown non-sofic, then in particular there is a non-sofic group, which is open. The embedding also shows that the mapping class group of a mixing SFT is not locally embeddable in finite groups, since $V$ is not (since it is finitely presented, infinite and simple).

\begin{theorem}
\label{thm:V}
Thompson's group $V$ embeds in the mapping class group of the vertex shift defined by the matrix $\left[\begin{smallmatrix}
1 & 1 & 0 \\
1 & 1 & 1 \\
1 & 1 & 1 \\
\end{smallmatrix}\right]$.
\end{theorem}

This is equivalently the subshift $X \subset \{0,1,2\}^\Z$ with the single forbidden pattern $02$.

\begin{theorem}
\label{thm:2V}
The Brin-Thompson group $2V$ embeds in the mapping class group of the vertex shift defined by the matrix $\left[\begin{smallmatrix}
      1 & 1 & 1 & 0 & 0 & 0 \\
      1 & 1 & 1 & 0 & 0 & 0 \\
      0 & 0 & 0 & 1 & 1 & 1 \\
      0 & 0 & 0 & 1 & 1 & 0 \\
      0 & 0 & 0 & 1 & 1 & 1 \\
      0 & 1 & 1 & 0 & 0 & 0 \\
\end{smallmatrix}\right]$.
\end{theorem}

\section{Definitions}

We use standard terminology and notation for words and (regular) languages. The \emph{empty word} is $\epsilon$. See \cite{LiMa95} for definitions of SFTs and vertex shifts. Briefly, a \emph{subshift} is a set of bi-infinite words (or \emph{points} or \emph{configurations}) defined by a set of forbidden finite subwords, an \emph{SFT} is one where this set is finite, and a \emph{vertex shift} is one where the words all have length $2$. A vertex shift can be seen as the set of bi-infinite paths in a graph (which we represent by its adjacency matrix). The \emph{language} of a subshift is the set of finite subwords of its configurations. An SFT is \emph{mixing} if for some $n$, for any two $u, v$ that appear in its language, also $uwv$ does, for some word $w$ with $|w| = n$.

The following is Thompson's group $V$.

\begin{definition}
\label{def:V}
Thompson's group $V$ is the group of homeomorphisms $f : \{0,1\}^\N \to \{0,1\}^\N$ such that for some $n \in \N$, there exists a function $F : \{0,1\}^n \to \{0,1\}^*$ such that $f(ux) = F(u)x$ for all $u \in \{0,1\}^n, x \in \{0,1\}^\N$.
\end{definition}

The following is the Brin-Thompson $2V$ \cite{Br04a}.

\begin{definition}
\label{def:2V}
The Brin-Thompson group $2V$ is the group of homeomorphisms $f : (\{0,1\}^\N)^2 \to (\{0,1\}^\N)^2$ such that for some $n \in \N$, there exist functions $F_i : (\{0,1\}^n)^2 \to \{0,1\}^*$ such that $f((ux, vy)) = (F_1(u, v)x, F_2(u, v)y)$ for all $u,v \in \{0,1\}^n, x,y \in \{0,1\}^\N$.
\end{definition}

Let $X$ be a subshift. The \emph{mapping torus} $\mathbf{S}X$ is the quotient of $X \times \R$ by the equivalence relation $\cong$ defined by $(x, t) \cong (\sigma^n(x), t-n)$ for each $x \in X$, $n \in \Z$. Write $[x, t]$ for the equivalence class of $(x,t)$. The path components of this space are the orbits of the $\R$-flow $r + [x, t] = [x, t+r]$, i.e.\ the equivalence classes of the sets $\{x\} \times \R$. Self-homeomorphisms of $\mathbf{S}X$ map orbits onto orbits. Just like in the case of subshifts, points of the mapping torus are called \emph{points} or \emph{configurations}.

Write $\mathcal{F}(X)$ for the subgroup of homeomorphisms which preserve the natural orientation of each orbit. For $f,g \in \mathcal{F}(X)$ say $f$ and $g$ are \emph{isotopic} if there exists a continuous map $h: [0,1] \times \mathbf{S}X \to \mathbf{S}X$ such that $f(x) = h(0,x)$ and $g(x) = h(1,x)$ for all $x \in \mathbf{S}X$, and $x \mapsto h(t,x)$ is in $\mathcal {F}(X)$ for all $t \in [0,1]$. The following is the Boyle-Chuysurichay mapping class group \cite{Bo02a,BoCh17}.

\begin{definition}
Let $X$ be a subshift. The \emph{mapping class group} $\MCG(X)$ is the quotient of the group $\mathcal{F}(X)$ under isotopy.
\end{definition}

We think of an element of the mapping torus as a concatenation (without gaps) of length-$1$ intervals labeled with a symbol from the alphabet, and the flow is the shift map on such unions. The word \emph{tile} refers to a tuple of consecutive such intervals. Concretely, tiles represent (finite subintervals of) paths in the mapping torus, whose labels traverse the corresponding cylinder in the subshift, or more abstractly the thicket of all paths traversing this cylinder. In our proofs, we refer to tiles by simply writing the word read from the labels of the intervals. When thinking of words like this, we refer to their subwords (with positions) as \emph{pieces}, with the understanding that the ``piece'' $v$ in $uvw$ refers to the geometric subtile corresponding to the word $v$, starting at coordinate $|u|$, even if $v$ appears multiple times in $uvw$.

In practice, we define elements of the mapping class group by describing the ``rewrites'' that are performed in a configuration of the mapping torus. We use a discrete rule to split a configuration into a concatenation of tiles, and then use a rule to determine what tiles we replace them by. If the replacements use words of different lengths, we interpolate linearly, i.e.\ the flow over a tile is decelerated or accelerated on the image side; formally, this simply means that the cocycle of the mapping class group element takes on nontrivial values. We refer to this as \emph{distortion}. See \cite{BoCh17} and its references for more details on discretization of mapping class groups.

\section{Veelike actions}

\begin{definition}
Let $L \subset A^*$ be a language and $G$ a group. An action $G \curvearrowright L$ is \emph{veelike} (or $G$ acts veelike) if for all $g \in G$ there exists $n \in \N$ such that for all $u \in A^*$ with $|u| = n$ there exists $u' \in A^*$ such that $guv=u'v$ for all $uv \in L$.
\end{definition}

This description determines a function $F : A^n \to A^*$ called the \emph{local rule} of $g$.

\begin{theorem}
\label{thm:Veelike}
Suppose $G$ admits a faithful veelike action on $L \subset A^*$. Let $X \subset (A \cup \{\#\})^\Z$ be the smallest subshift containing every configuration of the form
\[ \ldots w_{-2} \# w_{-1} \# w_0 \# w_1 \# w_2 \# \ldots \]
with $w_i \in L$. Then $G$ embeds in $\MCG(X)$.
\end{theorem}

\begin{proof}
Let $g \in G$, and $F : A^n \to A^*$ be the local rule for some $n \in \N$. For all $u \in A^*$ with $|u| \leq n$, inside every block $\#u\#$ map the piece $\#u$ onto a word $\# g \cdot u$, stretching linearly. The leftmost point of the leftmost $\#$-piece is called an \emph{anchor}. Configurations where the anchor is at the origin are not fixed (no non-trivial mapping class group element on $X$ fixes an open set if $L$ is infinite), but intuitively we think of these positions as fixed: rewrites do not touch these positions, and the linear distortion in the flow from rewrites with different word lengths does not affect them either.

For $uv \in A^*$ with $\#uv$ where $|u| = n$, inside occurrences of $\#uv\#$ map the piece $\#u$ linearly onto $\# (g  \cdot u)$. Note that $u$ or $g \cdot u$ may be the empty word, but this poses no problem, as the words $\# u$ and $\# (g \cdot u)$ nevertheless have positive length. Finally, fix all $A$-symbols to the left of which no anchor appears in the preceding $n+1$ steps in the flow. This rule determines a mapping class group element $\hat g \in \mathcal{F}(X)$

This mapping is clearly bijective on the flow orbits: to the right of any $\#$ we exactly simulate the action of $f$ (and when $\#$s do not appear on the left, the configuration is fixed). The map $g \mapsto \hat g$ is faithful because on the periodic flow orbit with repeating pattern $\# u$, $u \in A^*$, the action of $\hat G$ simulates the orbit of $u$ in the action of $G$. It is a homomorphism because if $g_n \circ \cdots \circ g_1 = \ID_L$ then $\hat g_n \circ \cdots \circ \hat g_1$ performs, on the level of symbols, the identity transformation between any two $\#$s (because the $V$-action on finite-support configurations cancels) and also to the right of every rightmost $\#$-symbol.

There can be a leftover distortion in the flow, but we can use the occurrences of the anchors to indeed ``anchor'' an isotopy from $\hat g_n \circ \cdots \circ \hat g_1$ to the identity map, i.e.\ our flow can act independently in the segment to the right of an anchor (up to the next anchor or until we reach content that was not reached by the rewrites), and we simply observe that the set of self-homeomorphisms of the interval $[0,r]$ fixing $0$ is path connected.
\end{proof} 

In the case where $X$ happens to be a mixing SFT (like in our application), one need not construct the isotopy explicitly, as it is known that any mapping class group element that fixes all flow orbits is isotopic to the identity \cite[Theorem~3.19]{BoCh17}, and it is much easier to see that flow orbits are fixed than to write out the explicit formulas for the isotopy.

\begin{theorem}
\label{thm:VlikeV}
Thompson's group V admits a faithful veelike action on the regular language $L = \epsilon + (0 + 1)^*1$.
\end{theorem}

\begin{proof}
Let $X_0 = \{x \in \{0,1\}^\N \;|\; \sum_i x_i < \infty\}$. Define a map $\phi : L \mapsto X_0$ by
\[ \phi(\epsilon) = 0^\M, \;\; \phi(w) = w 0^\N, \]
and observe that this is a bijection. The formula $f \cdot w = \phi^{-1}(f(\phi(w)))$ defines an action of V on $L$. This action is veelike: suppose $f \in V$, and let $n \in \N$ and $F$ be as in Definition~\ref{def:V}. If $|w| \geq n+1$ and $w \in L$ then $w = uv1$ for some $|u| = n$. Then $f \cdot w = \phi^{-1}(f(uv10^\N)) = \phi^{-1}(F(u)v10^\N) = F(u)v1$. Thus, picking the same $n$ and $u' = F(u)$ for all $u \in \{0,1\}^n$, we obtain that the action is veelike. Faithfulness is obvious.
\end{proof}

\section{Veelike actions on pair languages}

With only small notational changes, we can  generalize the previous section to ``pair languages'', i.e.\ languages of pairs of words.

A \emph{pair language} is a subset of the \emph{word pairs} $A^* \times B^*$, where $A$ and $B$ are finite alphabets. Let $\alpha_n : A^* \times B^* \to A^* \times B^*$ be the prefix extraction map $\alpha_n(ux, vy) = (u, v)$ where $|u| \leq n \wedge (|u| = n \vee x = \epsilon)$, and $|v| \leq n \wedge (|v| = n \vee y = \epsilon)$. Write $A^{\leq n} = \{w \in A^* \;|\; |w| \leq n\}$. We have $\alpha_n(A^* \times B^*) = A^{\leq n} \times B^{\leq n}$. Concatenation of words pairs is performed coordinatewise. Define $\omega_n : A^* \times B^* \to A^* \times B^*$ by taking the tails after removing the $\alpha_n$-prefix, i.e.\ $\omega_n(u, v) = (u',v') \iff \alpha_n(u, v) \cdot (u', v') = (u,v)$.

\begin{definition}
Let $L \subset A^* \times B^*$ be a pair language and $G$ a group. An action $G \curvearrowright L$ is \emph{veelike} (or $G$ acts veelike) if for all $g \in G$ there exists $n \in \N$ and a \emph{local rule} $F : L \cap (A^{\leq n} \times B^{\leq n}) \to A^* \times B^*$ such that $g \cdot (u, v) = F(\alpha_n(u, v)) \cdot \omega_n(u)$, where by $F(\alpha_n(u, v))$ we mean the two components of $\psi_n(u, v)$ are given as parameters to $F$.
\end{definition}

Denote by $u^R$ the reversal of words, i.e.\ the function satisfying $a^R = a$ for $|a| \leq 1$ and $(u \cdot v)^R = v^R \cdot u^R$ for all words $u, v$.

\begin{theorem}
\label{thm:Veelike2}
Suppose $G$ admits a faithful veelike action on a pair language $L \subset A^* \times B^*$ where $A \cap B = \emptyset$. Let $X \subset (A \cup B \cup \{\#, @\})^\Z$ be the smallest subshift containing every configuration of the form
\[ \ldots u_{-1} @ v_{-1} \# u_0 @ v_0 \# u_1 @ v_1 \# u_2 @ v_2 \ldots \]
with $(u_i^R, v_i) \in L$, where $(u_i, v_i) \in L$. Then $G$ embeds in $\MCG(X)$.
\end{theorem}

\begin{proof}
This is exactly analogous to Theorem~\ref{thm:Veelike}. Near each occurrence of $@$, read the (beginnings of the) words $u_i$ and $v_i$ on the left and right, for $n$ steps or until another $\#$ is reached. Perform the substitution given by the local rule $F$ and interpolate linearly, the midpoint of the $@$-piece taking the role of the special fixed point.
\end{proof}

\begin{theorem}
\label{thm:Vlike2V}
The Brin-Thompson group $2V$ admits a faithful veelike action on the pair language $(\epsilon + (0 + 1)^*1)^2$.
\end{theorem}

\begin{proof}
This is exactly analogous to the case of $V$. We simply identify a pair of words $(u, v)$ with $(\phi(u), \phi(v))$, where $\phi : L \to X_0$ is the map from the proof of Theorem~\ref{thm:VlikeV}, and conjugate the natural action of $2V$ through this bijection.
\end{proof}

\section{The main results}

The language where $V$ acts is not just regular, but it is \emph{locally $2$-testable}, meaning the inclusion of a word depends only on its subwords and its prefixes and suffices of length $2$. In this case, the subshift in the statement of Theorem~\ref{thm:Veelike} is clearly of finite type. Now Theorem~\ref{thm:V} is a matter of composing the lemmas. We include a technical fact about the embedding, which is of symbolic dynamical interest. See \cite{BoCh17} for the definition of the Bowen-Franks representation.

\begin{theorem}
\label{thm:V}
Thompson's V embeds in the mapping class group of the vertex shift defined by the matrix $\left[\begin{smallmatrix}
1 & 1 & 0 \\
1 & 1 & 1 \\
1 & 1 & 1 \\
\end{smallmatrix}\right]$. The image of this embedding is in the kernel of the Bowen-Franks representation.
\end{theorem}

\begin{proof}
The vertex shift is simply the subshift from Theorem~\ref{thm:Veelike} when applied to the language defined in Theorem~\ref{thm:VlikeV}, if we set $\# = 2$ and columns and rows are indexed with $0,1,2$ in this order. To see that the image is in the kernel of the Bowen-Franks representation, simply observe that it is, by definition, some representation of the mapping class group by automorphisms of a finitely-generated abelian group. Finitely-generated abelian groups are residually finite, and the automorphism group of a finitely-generated residually finite group is residually finite \cite{Ba63}, so the finitely-generated infinite simple group $V$ must have a trivial representation.
\end{proof}

\begin{theorem}
\label{thm:2V}
The Brin-Thompson group $2V$ embeds in the mapping class group of the vertex shift defined by the matrix $\left[\begin{smallmatrix}
      1 & 1 & 1 & 0 & 0 & 0 \\
      1 & 1 & 1 & 0 & 0 & 0 \\
      0 & 0 & 0 & 1 & 1 & 1 \\
      0 & 0 & 0 & 1 & 1 & 0 \\
      0 & 0 & 0 & 1 & 1 & 1 \\
      0 & 1 & 1 & 0 & 0 & 0 \\
\end{smallmatrix}\right]$. The image of this embedding is in the kernel of the Bowen-Franks representation.
\end{theorem}

\begin{proof}
This exactly analogous to the case of $V$. Rename the symbols $0,1$ in the leftmost component of $(\epsilon + (0 + 1)^*1)^2$ from Theorem~\ref{thm:Vlike2V} to $0_A, 1_A$ and the ones in the rightmost component to $0_B, 1_B$. Now taking the dimensions to represent $0_A, 1_A, @, 0_B, 1_B, \#$ in this order, the subshift in Theorem~\ref{thm:Veelike2} is the vertex shift corresponding to this matrix. Again the fact the image is in the kernel of the Bowen-Franks representation is automatic, since $2V$ is a finitely-generated infinite simple group.
\end{proof}

The symbol $\#$ is not really used, but we feel this additional separator clarifies the situation. The separator $@$ inside a word pair is important, since the anchor needs to be inside a symbol that is never removed in situations where either the leftmost or rightmost word becomes empty, for linear stretching to be meaningful.

\section{Discussion and questions}

Given a language $L \subset A^*$, one can consider the \emph{veelike group of $L$}, the largest group that faithfully acts veelike on it, namely the group of all bijections $g : A^* \to A^*$ such that there exists $n \in \N$ such that for all $u \in A^*$ with $|u| = n$ there exists $u' \in A^*$ such that $guv=u'v$ for all $uv \in L$.

Thompson's $V$ embeds in the veelike group of $(0+1)^*1$ by Theorem~\ref{thm:VlikeV} but this embedding is not surjective onto the veelike group, since for example the involution mapping $\epsilon \leftrightarrow 1$ (and fixing other words) does not come from Thompson's $V$. It does not seem easy to determine the soficity of the veelike group on any regular language of exponential growth, in particular that of $A^*$ seems closely related to Thompson's $V$.

\begin{theorem}
The mapping class group of every positive entropy sofic shift $X$ contains the veelike group of $A^*$, for every finite alphabet $A$.
\end{theorem}

\begin{proof}
Pick mutually unbordered words $u_\#$ and $u_a$ for $a \in A$ in the language of $X$, all of the same length, so that each of these words maps to the same element of the syntactic monoid, and $u_\#^*$ (thus any concatenation of the words $u_i$) is contained in the language of $X$. The existence of such a set of words follows by standard arguments. The proof of Theorem~\ref{thm:Veelike} goes through almost directly, pretending the word $u_i$ is the corresponding letter $i$, and fixing the contents of a configuration when no $u_i$-word is nearby. There are some subtleties when the coding breaks, i.e.\ when a concatenation of $u_i$s is followed on the left or right by something other than another word $u_j \in U$. If the coding breaks on the left, we simply do not modify the configuration. If the coding breaks on the right, then we act as if this breaking point begins another $u_\#$-word.
\end{proof}

Due to the above discussion, although we do not know how to embed $V$, it seems likely that the soficity of the mapping class of any positive entropy sofic shift is difficult to determine.

We suspect that the veelike group of every regular language of polynomial growth is sofic, and that the mapping class group of every countable sofic shift is sofic (the automorphism group of every countable sofic shift, as a topological dynamical system, is known to be amenable).

\begin{question}
Does Thompson's $V$ embed in the full veelike group of $A^*$ for a finite alphabet $A$?
\end{question}

One can ask whether the Brin-Thompson group $2V$ (or higher groups $nV$) embeds in the veelike group of $A^*$ for a finite alphabet $A$. One can of course formulate a natural notion of the veelike group of a pair language (or even an $n$-tuple language), and we do not know whether $nV$ can act faithfully on a $k$-tuple language for $k < n$.

\begin{question}
Does Thompson's $V$ embed in the mapping class group of a full shift? Does the Brin-Thompson $2V$?
\end{question} 

\begin{question}
Does the higher-dimensional Thompson groups $3V$ embed in the mapping class group of a mixing SFT? Higher $nV$?
\end{question}

\section*{Acknowledgements}

We thank Mike Boyle for several useful comments and corrections on a preliminary draft. The author was supported by Academy of Finland project 2608073211.

\bibliographystyle{plain}
\bibliography{../../../bib/bib}{}

\end{document}